\documentclass[12pt,reqno]{amsart}
\usepackage{amscd,amssymb,amsthm}
\usepackage{graphicx}

\usepackage{enumerate}
\theoremstyle{plain}
\newtheorem{theorem}{Theorem}[section]
\newtheorem{corollary}[theorem]{Corollary}
\newtheorem{lemma}[theorem]{Lemma}
\newtheorem{proposition}[theorem]{Proposition}
\theoremstyle{remark}
\newtheorem*{remark}{Remark}

\newcommand{\reel}{\mathbb{R}}
\newcommand{\nat}{\mathbb{N}}

\newcommand{\vf}{\varphi}
\newcommand{\eps}{\varepsilon}
\newcommand{\abs}[1]{\left\vert #1\right\vert }

\newcommand{\N}[1]{\muskip0=-2mu{\left|\mkern\muskip0\left|#1\right|\mkern\muskip0\right|}}

\newcommand{\bg}{\medskip\goodbreak}

\newcommand{\vers}{{\longrightarrow}}

\newenvironment{enumeratea}{\begin{enumerate}%
	[\upshape (a)]}{\end{enumerate}}
\newenvironment{enumeratei}{\begin{enumerate}%
	[\itshape i.]}{\end{enumerate}}

\title[An Inequality for Bounded Functions]
{An Inequality for Bounded Functions}
\author[Omran Kouba]{Omran Kouba$^\dag$}
\address{Department of Mathematics \\
Higher Institute for Applied Sciences and Technology\\
P.O. Box 31983, Damascus, Syria.}
\email{omran\_kouba@hiast.edu.sy}
\thanks{$^\dag$ Department of Mathematics, Higher Institute for Applied Sciences and Technology.}

\begin{document}
\parindent=0pt
%\date{\today}
\begin{abstract}
In this note we prove optimal inequalities for bounded functions in terms 
of their deviation from their mean. These results extend and generalize some known
 inequalities due to Thong (2011) and Perfetti (2011). \par
\end{abstract}
\smallskip\goodbreak

\parindent=0pt
\maketitle
%%%%%%%%%%%%%%%%%%%%%%%%
\section{\bf Introduction }\label{sec1}
\bg
\parindent=0pt
\qquad Let $L^\infty([0,1])$ be the space of essentially bounded measurable real functions on $[0,1]$ equipped with
the well-known essential supremum norm $\N{\cdot}_\infty$, and
consider two real numbers $m$ and $M$ such that $m<0<M$.
Let ${\mathcal F}_{m,M}$ denote the closed subset of $L^\infty([0,1])$ consisting of functions
$f:[0,1]\vers\reel$ such that $m\leq f\leq M$ and $\int_0^1f(x)\,dx=0$. 

\begin{equation}
{\mathcal F}_{m,M}=\left\{f\in L^\infty([0,1]): m\leq f\leq M \hbox{ and }\int_0^1f(x)\,dx=0\right\}.
\end{equation}

For $f$ in $L^\infty([0,1])$ we define the continuous function $J(f):[0,1]\vers \reel$ by
\begin{equation}
\forall\,x\in [0,1],\qquad J(f)(x)=\int_0^xf(t)\,dt.
\end{equation}

\qquad In \cite{thong} it was asked to show that for every \textit{continuous} $f$ that belongs to
${\mathcal F}_{m,M}$ one has the following inequality :
\begin{equation}\label{E:eq1}
\abs{\int_0^1xf(x)\,dx}\leq\frac{1}{2}\cdot\frac{-mM}{M-m}
\end{equation}
Noting that for continuous functions $f$ from ${\mathcal F}_{m,M}$ we have
\begin{align*}
\int_0^1xf(x)\,dx&=\int_0^1x(J(f))'(x)\,dx\\
&=\big[xJ(f)(x)\big]_{x=0}^{x=1}-\int_0^1J(f)(x)\,dx\\
&=-\int_0^1J(f)(x)\,dx
\end{align*}
We see that \eqref{E:eq1} would follow from the stronger inequality
\begin{equation}\label{E:eq2}
\int_0^1\abs{J(f)(x)}\,dx\leq\frac{1}{2}\cdot\frac{-mM}{M-m}.
\end{equation}
\qquad Also it was asked in  \cite{perfetti} to prove that for every $f$ in ${\mathcal F}_{m,M}$
one has
\begin{equation}\label{E:eq3}
\int_0^1\big(J(f)(x)\big)^2\,dx\leq \frac{-mM}{6(M-m)^2}(3M^2-8mM+3m^2).
\end{equation}
but in \cite{kou} the following sharper result was proved
\begin{equation}\label{E:eq4}
\left(\int_0^1\big(J(f)(x)\big)^2\,dx\right)^{1/2}\leq \frac{1}{\sqrt{3}}\cdot\frac{-mM}{M-m},
\end{equation}
and the cases of equality were caracterized.\bg
\qquad In this note we will generalize these results to give sharp bounds in terms of
$m$, $M$ and $\vf$ for $\int_0^1\vf\big(\abs{J(f)(x)}\big)\,dx$,
where $\vf$ is an increasing function, and we will caracterize the cases of equality.

As corollaries we will prove that for functions $f$ in
${\mathcal F}_{m,M}$, we have
\begin{align*}
\left(\int_0^1\abs{J(f)(x)}^p\,dx\right)^{1/p}& \leq \frac{1}{\root{p}\of{1+p}}\cdot\frac{-mM}{M-m},\quad \hbox{for $p>0$.}\\
\noalign{and}\\
\exp\left(\int_0^1\log\abs{J(f)(x)}\,dx\right)&\leq \frac{1}{e}\cdot\frac{-mM}{M-m}.
\end{align*}
\bigskip\goodbreak
%%%%%%%%%%%%%%%%%%%%%
\section{\bf The Main Results }\label{sec2}
\bg
Clearly we have the following simple property :
\bg
\begin{proposition}
For every $f\in {\mathcal F}_{m,M}$ we have
$$\N{J(f)}_\infty\leq \frac{-mM}{M-m}.$$
\end{proposition}

\begin{proof} Indeed, consider $f\in {\mathcal F}_{m,M}$ and $x\in[0,1]$. We distinguish two cases :\bg
\parindent20pt
\begin{enumeratei}
\item   $x\in\left[0, \frac{-m}{M-m}\right]$. Since $f(t)\leq M$ for $t\in[0,x]$ we deduce
that
\[
J(f)(x)=\int_0^xf(t)\,dt\leq M x\leq \frac{-mM}{M-m}.
\]
\item  $x\in\left[\frac{-m}{M-m},1\right]$. Here we have $-f(t)\leq -m$ for $t\in[x,1]$ so
\[
J(f)(x)=\int_x^1(-f)(t)\,dt\leq -m(1-x)\leq \frac{-mM}{M-m}.
\]
\end{enumeratei}
So we have shown that for every $f\in {\mathcal F}_{m,M}$ we have
\begin{equation}\label{E:eq7}
\forall\,x\in[0,1],\qquad J(f)(x)\leq \frac{-mM}{M-m}.
\end{equation}

Applying \eqref{E:eq7} to $-f\in{\mathcal F}_{-M,-m}$ we conclude also that
\begin{equation}\label{E:eq8}
\forall\,x\in[0,1],\qquad -J(f)(x)\leq \frac{-mM}{M-m}.
\end{equation}
Now, from \eqref{E:eq7}  and \eqref{E:eq8}, we arrive to the conclusion that
\[
\forall\,x\in[0,1],\qquad \abs{J(f)(x)}\leq \frac{-mM}{M-m},
\]
as desired.
\end{proof}
\bg
\qquad The next lemma is a well-known result on convex functions, (See for example \cite[Ch 4]{gar}.)  But since
its statement is somehow unusual, we will include a proof for the convenience of the reader.
\bg
\begin{lemma}\label{lm1}
\sl Let $\vf:[0,T]\vers\reel$ be a monotonous increasing function which is not constant on $(0,T)$. 
For $t\in(0,T]$ we define $K(\vf,t)$ by
$$K(\vf,t)=\frac{1}{t}\int_0^t\vf(x)\,dx.$$
Then, for all $t\in(0,T)$ we have $K(\vf,t)<K(\vf,T)$.
\end{lemma}
\bg
\begin{proof} 
Indeed, for $\alpha\in (0,1)$ we have
\[
K(\vf,\alpha T)=\frac{1}{\alpha T}\int_0^{\alpha T}\vf(x)\,dx=\frac{1}{T}\int_0^T\vf(\alpha u)\,du.
\]
So, if $0<\alpha<1$ then
\[
K(\vf,T)-K(\vf,\alpha T)=\frac{1}{T}\int_0^T\big(\vf(u)-\vf(\alpha  u)\big)\,du\geq0.
\]
The last inequality follows from the fact that $u\mapsto\big(\vf(u)-\vf(\alpha  u)\big)$ 
is nonnegative on $[0,T]$ 
because $\vf$ is increasing.

\bg
\qquad
Now suppose that we have $K(\vf,T)=K(\vf,\alpha T)$ for some $\alpha\in(0,1)$. This implies that the set
\[
{\mathcal S}=\{u\in [0,T]:\vf(u)=\vf(\alpha  u)\}
\]
has Lebesgue measure equal to $\lambda([0,T])=T$. It follows that the set
\[
{\mathcal S}^\prime=\bigcap_{n\geq 1}\big(\alpha^{-n}{\mathcal S}\big)
\]
 has also
 Lebesgue measure equal to $T$. In particular, ${\mathcal S}^\prime$ is a dense subset of $(0,T)$. Now, consider $u\in {\mathcal S}^\prime$.
We have $\vf(\alpha^ku)=\vf(\alpha^{k+1}u)$ for every $k\geq0$. Thus, for every $k\geq0$ we have
$\vf(u)=\vf(\alpha^ku)$, so letting $k$ tend to $+\infty$ we obtain
$\vf(u)=\vf(0^+)$. Since ${\mathcal S}^\prime$ is a dense subset of $(0,T)$, there is an increasing sequence $(u_n)_{n\geq0}$ in ${\mathcal S}^\prime$
that converges to $T$, thus $\vf(0^+)=\lim_{n\to\infty}\vf(u_n)=\vf(T^-)$. This means that
$\vf$ is constant on $(0,T)$ which is contrary to the hypothesis. So we must have $K(\vf,T)>K(\vf,\alpha T)$ for
every $\alpha\in(0,1)$ and the proof of the Lemma is complete.
\end{proof}

\bg
The next theorem is the main result of this note :
\bg
\begin{theorem}\label{th1}
Let $\vf$ be a positive monotone increasing function on $:\left[0,\frac{-mM}{M-m}\right]$. For every $f\in{\mathcal F}_{m,M}$ we have
\[
\int_0^1\vf(\abs{J(f)(x)})dx\leq K\left(\vf, \frac{-mM}{M-m}\right),
\]
where $K(\,.\, ,\,.\,)$ is defined in Lemma \ref{lm1}. Moreover,  if $\vf$ is not constant on $\left(0,\frac{-mM}{M-m}\right)$, then
equality holds if and only if $f$ co\"\i nsides for almost every $x$ in $[0,1]$ with one of the functions $f_0$ or $f_1$ defined  by
\[
f_0(x)=\left\{
\begin{matrix}
M&\hbox{ if }&x\in\left[0,\frac{-m}{M-m}\right),\\
m&\hbox{ if }&x\in\left[\frac{-m}{M-m},1\right].
\end{matrix}
\right.\qquad
f_1(x)=\left\{
\begin{matrix}
m&\hbox{ if }&x\in\left[0,\frac{M}{M-m}\right),\cr
M&\hbox{ if }&x\in\left[\frac{M}{M-m},1\right].
\end{matrix}
\right.
\]
\end{theorem}
\bg
\begin{proof}
 Since $f$ is integrable, $J(f)$ is continuous on $[0,1]$. If $J(f)=0$,
 ({\it i.e.} $f=0~a.e.$,) there is nothing to be proved. So,
in what follows we will suppose that $J(f)\ne 0$.\bg
\qquad The continuity of $J(f)$  shows that the set ${\mathcal O}=\{x\in(0,1): J(f)(x)\ne 0\}$
is an open set. Moreover, since $J(f)(0)=J(f)(1)=0$, we see that $J(f)(t)=0$ for every $t\in[0,1]\setminus{\mathcal O}$. 
\bg
\qquad The open set ${\mathcal O}$ is the union of at most denumbrable family of disjoint open
intervals. Thus there exist ${\mathcal N}\subset\nat$ and a family $(I_n)_{n\in{\mathcal N}}$ of non-empty
{\it disjoint} open sub-intervals of $(0,1)$  such that  ${\mathcal O}=\cup_{n\in{\mathcal N}}I_n$.\par
\qquad Suppose that $I_n=(a_n,b_n)$. Since $a_n$ and $b_n$ belong to $[0,1]\setminus{\mathcal O}$, we conclude
that $J(f)(a_n)=J(f)(b_n)=0$, while $J(f)$ keeps a constant sign on $I_n$.
So, let us consider two cases :
\bg
\begin{enumeratea}
\item  $J(f)(x)>0$ for $x\in I_n$. From the inequality $m\leq f\leq M$ we conclude that, 
for $x\in I_n$, we have
\begin{equation}\label{E:eq9}
J(f)(x)=J(f)(x)-J(f)(a_n)=\int_{a_n}^xf(t)\,dt\leq M(x-a_n)
\end{equation}
and
\begin{equation}\label{E:eq10}
J(f)(x)=-(J(f)(b_n)-J(f)(x))=\int_x^{b_n}(-f)(t)\, dt\leq-m(b_n-x)=m(x-b_n).
\end{equation}
Combining \eqref{E:eq9} and \eqref{E:eq10} we obtain
\[
\forall\,x\in I_n,\qquad 0<J(f)(x)\leq \min(M(x-a_n),m(x-b_n)),
\]
and consequently, using the definition of $K(\,.\, ,\,.\,)$ from Lemma \ref{lm1}, we obtain
\begin{align}\label{E:eq12}
\int_{I_n}\vf(\abs{J(f)(x)})\,dx&\leq\int_{a_n}^{b_n}\vf\big(\min(M(x-a_n),m(x-b_n))\big)\,dx\notag\\
&= \int_{a_n}^{a_n-m(b_n-a_n)/(M-m)}\vf(M(x-a_n))dx+\notag\\
&\qquad\int_{b_n-M(b_n-a_n)/(M-m)}^{b_n}\vf(m(x-b_n))\,dx\notag\\
&= \frac{1}{M}\int_0^{-mM(b_n-a_n)/(M-m)}\vf(t)dt+\frac{1}{-m}\int_0^{M(b_n-a_n)/(M-m)}\vf(t)dt\notag\\
&= (b_n-a_n)K\left(\vf,\frac{-mM(b_n-a_n)}{M-m}\right)
\end{align}
with equality if and only if $J(f)(x)= \min(M(x-a_n),m(x-b_n))$ for every $x\in I_n$, that is, if and only if,
$f(x)=M$ for almost every $x\in\left[a_n,\frac{Ma_n-mb_n}{M-m}\right)$, and 
$f(x)=m$ for almost every $x\in\left[\frac{Ma_n-mb_n}{M-m},b_n\right]$.
\bg
\item  $J(f)(x)<0$ for $x\in I_n$. From $m\leq f\leq M$ we conclude that, for $x\in I_n$, we have
\begin{equation}\label{E:eq13}
J(f)(x)=J(f)(x)-J(f)(a_n)=\int_{a_n}^xf(t)\,dt\geq m(x-a_n)
\end{equation}
and
\begin{equation}\label{E:eq14}
J(f)(x)=-(J(f)(b_n)-J(f)(x))=\int_x^{b_n}(-f)(t)\, dt\geq -M(b_n-x).
\end{equation}
Again, combining \eqref{E:eq13} and \eqref{E:eq14} we get
\[
\forall\,x\in I_n,\qquad 0<-J(f)(x)\leq \min(-m(x-a_n),M(b_n-x)),
\]
and consequently
\begin{align}\label{E:eq15}
\int_{I_n}\vf(\abs{J(f)(x)})\,dx&\leq\int_{a_n}^{b_n}\vf\big(\min(m(a_n-x),M(b_n-x))\big)\,dx\notag\\
&= \int_{a_n}^{a_n+M(b_n-a_n)/(M-m)}\vf(m(a_n-x))dx\notag\\
&\qquad+\int_{b_n+m(b_n-a_n)/(M-m)}^{b_n}\vf(M(b_n-x))\,dx\notag\\
&= \frac{1}{-m}\int_0^{-mM(b_n-a_n)/(M-m)}\vf(t)dt+\frac{1}{M}\int_0^{-mM(b_n-a_n)/(M-m)}\vf(t)dt\notag\\
&=(b_n-a_n)K\left(\vf,\frac{-mM(b_n-a_n)}{M-m}\right),
\end{align}
with equality if and only if $J(f)(x)= \max(m(x-a_n),M(x-b_n))$ for every $x\in I_n$, that is, if and only if,
$f(x)=m$ for almost every $x\in\left[a_n,\frac{Mb_n-ma_n}{M-m}\right)$, and 
$f(x)=M$ for almost every $x\in\left[\frac{Mb_n-ma_n}{M-m},b_n\right]$.
\end{enumeratea}

So, comparing \eqref{E:eq12} and \eqref{E:eq15}  we see that in both cases we have
\[
\int_{I_n}\vf(\abs{J(f)(x)})\,dx\leq \abs{I_n}\cdot K\left(\vf,\frac{-mM\abs{I_n}}{M-m}\right).
\]
Therefore, using Lemma \ref{lm1}, we can write

\begin{align*}
\int_0^1\vf(\abs{J(f)(x)})\,dx&=\sum_{n\in{\mathcal N}}\int_{I_n}\vf(\abs{J(f)(x)})\,dx\leq
\sum_{n\in{\mathcal N}}\abs{I_n}\cdot K\left(\vf,\frac{-mM\abs{I_n}}{M-m}\right)\\
&\leq K\left(\vf,\frac{-mM}{M-m}\right)\sum_{n\in{\mathcal N}}\abs{I_n}=
 K\left(\vf,\frac{-mM}{M-m}\right) \, \abs{O}\\
&\leq  K\left(\vf,\frac{-mM}{M-m}\right)
\end{align*}
which is the desired inequality.
\bg
\qquad Moreover, analyzing the case of equality, and using Lemma \ref{lm1}, we see  that it can occur if and only if ${\mathcal O}=(0,1)$
and $f(x)=f_0(x)~a.e.$ or $f(x)=f_2(x)~a.e.$, where $f_0$ and $f_1$ are the functions defined in the statement of the Theorem.
This concludes the proof.
\end{proof}
\bg
\qquad Let us give some corollaries. For a positive real $p$ and a  funtion $f:$ from ${\mathcal L}^\infty([0,1])$ we recall the notation
\[
\N{f}_p=\left(\int_0^1\abs{f(x)}^p\, dx\right)^{1/p}.
\]
The following corollary gives sharp bounds on $\N{J(f)}_p$ when $f\in{\mathcal F}_{m,M}$. This generalizes
the inequalities from \cite{kou} (corresponding to $p=2$) and \cite{thong} (corresponding to $p=1$).
\bg
\begin{corollary}\label{cor1}
Let $p$ be a positive real number. Then, for every $f\in{\mathcal F}_{m,M}$ we have
\[
\N{J(f)}_p\leq \frac{1}{(p+1)^{1/p}}\cdot \frac{-mM}{M-m},
\]
with equality if and only if $f$ co\"\i nsides for almost every $x$ in $[0,1]$ with one of the functions $f_0$ or $f_1$ defined  by
\[
f_0(x)=\left\{
\begin{matrix}
M&\hbox{ if }&x\in\left[0,\frac{-m}{M-m}\right),\\
m&\hbox{ if }&x\in\left[\frac{-m}{M-m},1\right].
\end{matrix}
\right.\qquad
f_1(x)=\left\{
\begin{matrix}
m&\hbox{ if }&x\in\left[0,\frac{M}{M-m}\right),\cr
M&\hbox{ if }&x\in\left[\frac{M}{M-m},1\right].
\end{matrix}
\right.
\] 
\end{corollary}
\begin{proof}
This follows from Theorem \ref{th1}, by choosing $\vf(x)=x^p$.
\end{proof}
\bg
\qquad Applying Theorem \ref{th1}, to the function $\vf_\eps(x)=\log(\eps+x)$ for $\eps>0$, and then letting $\eps$
tend to $0$ we obtain the following corollary :\bg
\begin{corollary}\label{cor2}
For every $f\in{\mathcal F}_{m,M}$ we have
\[
\exp\left(\int_0^1\log\abs{J(f)(x)}\,dx\right)\leq \frac{1}{e}\cdot \frac{-mM}{M-m}.
\]
\end{corollary}

\begin{remark}
 Note that Corollary \ref{cor2} follows also from Corollary \ref{cor1} by letting $p$ tend to $0$.
\end{remark}
\bg

%%%%%%%%%%%%%%%%%%%%%%%%%%%%%%%%%%


\begin{thebibliography}{9}
\setlength{\itemsep}{5pt}

\bibitem{gar}
Garling,~D. J. H.,
\emph{ Inequalities, A Journey into Linear Analysis}.
 Cambridge University Press,   2007.

\bibitem{kou}
Kouba,~O.,
\emph{ Solution To Problem 23}.
 MathProblems, \textbf{2}, Issue 1 (2012), pp. 61--63.

[ONLINE : \texttt{http://www.mathproblems-ks.com}].

\bibitem{perfetti}
Perfetti,~P.,
\emph{ Proposed Problem 23}. 
MathProblems, \textbf{1}, Issue 4 (2011), pp. 32--47.

[ONLINE : \texttt{http://www.mathproblems-ks.com}].

\bibitem{thong}
Thong,~D.~V.,
\emph{ Problem 11581}. 
American Mathematical Monthly,  \textbf{118}, No. 6, (2011). pp. 557.
doi:10.4169/amer.math.monthly.118.06.557.



\end{thebibliography}
\end{document}